\documentclass[12pt]{amsart}
 \usepackage{amsmath}
 \usepackage{amsthm}
 \usepackage{amssymb}
 \usepackage{latexsym}
 \usepackage{graphicx}
 \usepackage{multicol}
 \usepackage{mathrsfs}
 \usepackage{pstricks, pst-node}
 \usepackage[all]{xy}
    \SelectTips{cm}{10}     
    \everyxy={<2.5em,0em>:} 
 \usepackage{fancyhdr}      
 \newcounter{ctr}
 \theoremstyle{plain}
 \newtheorem{theorem}{Theorem}[section]
 
 \newtheorem*{lemma*}{Lemma}
 \newtheorem{lemma}[theorem]{Lemma}
 \newtheorem{corollary}[theorem]{Corollary}
 \newtheorem{proposition}[theorem]{Proposition}

 \theoremstyle{definition}
 \newtheorem{definition}[theorem]{Definition}

 \newtheorem{example}[theorem]{Example}

 \newcommand{\CC}{\ensuremath{\mathbb{C}}}

 \renewcommand{\H}{\ensuremath{\mathscr{H}}}
 
 \renewcommand{\hom}{\text{\rm Hom}}

  \newcommand{\idelm}{\ensuremath{id}}

 \renewcommand{\L}{\ensuremath{\mathscr{L}}}
 
 \newcommand{\lie}{\text{\rm Lie}}

 \newcommand{\RR}{\ensuremath{\mathbb{R}}}

 \newcommand{\ZZ}{\ensuremath{\mathbb{Z}}}

\renewcommand{\S}{\ensuremath{\mathcal{S}}}
\newcommand{\tsr}{\ensuremath{\otimes}}
\newcommand{\C}{\ensuremath{C^{\prime}}} 

\newcommand{\br}[1]{\ensuremath{\overline{#1}}}
\renewcommand{\u}{{\ensuremath{q^{1/2}}}}  
\newcommand{\ui}{\ensuremath{q^{-1/2}}} 
\newcommand{\uyw}{\ensuremath{(q^{1/2})^{\ell(y)-\ell(w_0)}}} 
\newcommand{\eH}{\ensuremath{{\widehat{\H}}}} 
\newcommand{\eW}{\ensuremath{{W_e}}}
\newcommand{\aW}{\ensuremath{{W_a}}}
\newcommand{\y}{\ensuremath{Y}} 
\newcommand{\leftexp}[2]{{\vphantom{#2}}^{#1}{#2}}
\newcommand{\leftsub}[2]{{\vphantom{#2}}_{#1}{#2}}
\newcommand{\lj}[2]{\ensuremath{{#1}_{#2}}}
\newcommand{\lJ}[2]{\ensuremath{{#1}^{#2}}}
\newcommand{\rj}[2]{\ensuremath{\leftsub{#2}{#1}}}
\newcommand{\rJ}[2]{\ensuremath{\leftexp{#2}{#1}}}
\newcommand{\alcove}{\ensuremath{\mathbf{A_0}}}
\newcommand{\bx}[1]{\ensuremath{\mathbf{B_{#1}}}}
\newcommand{\boxo}{\ensuremath{\bx{0}}}
\newcommand{\chamber}{\ensuremath{\mathbf{C_0}}}
\newcommand{\chamberf}{\ensuremath{\mathbf{C_{f0}}}}
\newcommand{\fcoroots}{\ensuremath{{R'_f}^{\!\!\vee}}}
\newcommand{\rsd}[1]{\ensuremath{\hat{#1}}}

\newcommand{\be}{\begin{equation}}
\newcommand{\ee}{\end{equation}}

\renewcommand{\ng}{\text{-}}

\begin{document}
\author{Jonah Blasiak}
\title{A factorization theorem for affine Kazhdan-Lusztig basis elements}

\begin{abstract}
The lowest two-sided cell of the extended affine Weyl group $W_e$ is the set $\{w \in W_e: w = x \cdot w_0 \cdot z, \text{ for some } x,z \in W_e\}$, denoted $W_{(\nu)}$. We prove that for any $w \in W_{(\nu)}$, the canonical basis element $\C_w$ can be expressed as $\frac{1}{[n]!} \chi_\lambda({\y}) \C_{v_1 w_0} \C_{w_0 v_2}$, where $\chi_\lambda({\y})$ is the character of the irreducible representation of highest weight $\lambda$ in the Bernstein generators, and $v_1$ and  $v_2^{-1}$ are what we call primitive elements.  Primitive elements are naturally in bijection with elements of the finite Weyl group $W_f \subseteq W_e$, thus this theorem gives an expression for any $\C_w$, $w \in W_{(\nu)}$ in terms of only finitely many canonical basis  elements.  After completing this paper, we realized that this result was first proved by Xi in \cite{X}.  The proof given here is significantly different and somewhat longer than Xi's, however our proof has the advantage of being mostly self-contained, while Xi's makes use of results of Lusztig from \cite{L Jantzen} and Cells in affine Weyl groups I-IV and the positivity of Kazhdan-Lusztig coefficients.
\end{abstract}

\maketitle

\section{Introduction}
  This work came about from a desire to better understand the polynomial ring $\CC[y_1,\ldots,y_n]$ as an $\CC\S_n$-module in a way compatible with the multiplicative structure of the polynomial ring. The hope was that a quantization of the polynomial ring and its $\S_n$ action would rigidify the structure and make a combinatorial description more transparent. This hope has largely been realized by the type $A$ extended affine Hecke algebra and its canonical basis and, indeed, this is the subject of the forthcoming paper \cite{B3}.

  While the use of crystal bases of quantum groups to do tableau combinatorics is well established and used prolifically, the connection between combinatorics and canonical bases of Hecke algebras is less developed. This may be because such combinatorics involves computing the weights $\mu$ of $W$-graph edges, which is difficult, or finding a way to determine cells that avoids such computation. The main theorem of this paper simplifies the description of some canonical basis elements, the combinatorial implications of which will be discussed in \cite{B3}.  We show that the canonical basis elements $\C_w$ of the extended affine Hecke algebra $\H(W_e)$, for $w$ in the two-sided cell $W_{(\nu)}$, can be expressed in terms of a finite subset of the canonical basis and symmetric polynomials in the Bernstein basis.  After completing this paper, we realized that this result was first proved by Xi in \cite{X}. This paper gives a different proof that is somewhat longer but relies on less machinery than Xi's.

  This theorem allows us to construct a $q$-analogue of the ring of coinvariants $\CC[y_1,\ldots,y_n]/(e_1,\ldots,e_n)$ with a canonical basis (see \cite{B3}), but this application does not use the full strength of this theorem.  It may also be possible to use the theorem to construct a $q$-analogue of the ring of coinvariants in other types.  In addition, this theorem may shed some light on computing Kazhdan-Lusztig polynomials, although this is not our main focus.

The remainder of this paper is organized as follows.  In \textsection \ref{s Weyl} we introduce the extended affine Weyl group and its Hecke algebra along with the Bernstein presentation and canonical bases. In \textsection \ref{section primitive} we define the \emph{primitive} elements of an extended affine Weyl group associated to a simply connected reductive algebraic group  $G$ over  $\CC$, which are in bijection with elements of the associated finite Weyl group  $W_f$. Finally, \textsection\ref{section Fact Theorem} is devoted to a proof of our main result (Corollary \ref{main corollary}), which expresses a canonical basis element $\C_w$, $w \in W_{(\nu)}$   in terms of the $\C_x$ for $x$ primitive.

\section{Weyl groups and Hecke algebras}
\label{s Weyl}
Here we introduce Weyl groups and Hecke algebras in full generality and then specialize to the affine case.  We also recall the important presentation of the extended affine Hecke algebra due to Bernstein and state a crucial theorem of Lusztig expressing certain canonical basis element in terms of the Bernstein generators.  This material is explained more fully in \cite{H}, and \cite{Hu} gives a good exposition of the more basic notions.

\subsection{}
\label{ss root system}
A \emph{root system} $(X, (\alpha_i), (\alpha^\vee_i))$ consists of a finite-rank free abelian group $X$, its dual $X^\vee := \hom(X,\ZZ)$, simple roots $\alpha_1, \ldots, \alpha_n \in X$, and simple coroots $\alpha^\vee_1, \ldots, \alpha^\vee_n \in X^\vee$ such that the  $n \times n$ matrix with $(i,j)$-th entry $\langle \alpha_j, \alpha^\vee_i \rangle$ is a generalized Cartan matrix. Assume that the root system is non-degenerate, i.e. the simple roots are linearly independent.

Let $W$ be the Weyl group of this root system and $S = \{s_1, \dots, s_{n}\}$ the set of simple reflections.
The group $W$ is the subgroup of automorphisms of the lattice $X$ (and of  $X^\vee$) generated by the reflections $s_i$. Let  $R, R_+, R_-, Q$  be the roots, positive roots, negative roots, and root lattice.


The \emph{dominant weights}  $X_+$ and the \emph{dominant regular weights}  $X_{++}$ are the cones in $X$ given by
\be \begin{array}{rl}
X_+ &= \{\lambda \in X : \langle \lambda, \alpha_i^\vee \rangle \geq 0 \text{ for all } i\}, \\
X_{++} &=  \{\lambda \in X : \langle \lambda, \alpha_i^\vee \rangle \geq 1 \text{ for all } i\}.
\end{array}\ee

The pair $(W,S)$ is a Coxeter group with length function $\ell$ and Bruhat order $\leq$.
The \emph{length} $\ell(w)$ of $w$ is the minimal $l$ such that $w=s_1\ldots s_l$ for some $s_i\in S$, also equal to $|R_- \cap w(R_+)|$. If $\ell(uv)=\ell(u)+\ell(v)$, then $uv = u\cdot v$ is a \emph{reduced factorization}. The notation $L(w), R(w)$ will denote the left and right descent sets of $w$.

It is often convenient to use the geometry and topology of the real vector space $X^\vee_\RR :=X^\vee \tsr_\ZZ \RR$. This space contains the \emph{root hyperplanes} $H_\alpha = \{x \in X^\vee_\RR : \langle \alpha,x \rangle = 0\}$. The connected components of $X^\vee_\RR - \bigcup_\alpha H_\alpha$ are \emph{Weyl chambers} and the \emph{dominant Weyl chamber} is $\chamber =\{x\in X^\vee_\RR : \langle \alpha,x \rangle > 0 \text{ for all } \alpha\in R_+\}$. Its closure is a fundamental domain for the action of $W$ on $X^\vee_\RR$.

Certain relations in the Bruhat order may be understood in several ways. The following are equivalent:
\be
\label{e hyperplane separate}
\begin{array}{rl}
\text{(i)}& s_\alpha w < w. \\
\text{(ii)}& \text{one of $\alpha$ and $w^{-1}(\alpha)$ is in $R_+$ and the other is in $R_-$.} \\
\text{(iii)}& \chamber \text{ and } w(\chamber) \text{ are on opposite sides of }H_\alpha.
\end{array}
\ee
For the equivalence of (i) and (ii), see \cite[\textsection5.7]{Hu}, while the equivalence of (ii) and (iii) follows from the identity  $\langle \alpha,w(\chamber) \rangle = \langle w^{-1}(\alpha),\chamber\rangle$, where for a set $Z \subseteq X^\vee_\RR,$  $\langle Z, \alpha \rangle$ is defined to be the set $\{\langle z, \alpha \rangle: z \in Z\}$.


\subsection{}
For any $J\subseteq S$, the \emph{parabolic subgroup} $W_J$ is the subgroup of $W$ generated by $J$. Each left (resp. right) coset of $wW_J$ (resp. $W_Jw$) contains an unique element of minimal length called a minimal coset representative. The set of all such elements is denoted $W^J$ (resp. $\leftexp{J}{W}$). For any $w \in W$, define $\lJ{w}{J}$, $\rj{w}{J}$ by
\be w=\lJ{w}{J} \cdot \rj{w}{J},\ \lJ{w}{J} \in W^J,\ \rj{w}{J} \in W_J.\ee
Similarly, define $\lj{w}{J}$, $\rJ{w}{J}$ by
\be w= \lj{w}{J} \cdot \rJ{w}{J},\ \lj{w}{J} \in W_J,\ \rJ{w}{J} \in \leftexp{J}W.\ee

\subsection{}
\label{ss w0}
Any finite Weyl group $W_f$ contains a unique longest element $w_0$. The action of $w_0$ on $R$ satisfies $w_0(\alpha_i) = -\alpha_{d(i)}$ for some automorphism $d$ of the Dynkin diagram. Left (or right) multiplication by $w_0$ induces a Bruhat order-reversing involution on $W_f$ and therefore takes elements of length $l$ to elements of length $\ell(w_0)-l$. In particular, it takes the elements of length $\ell(w_0)-1$ to the simple reflections. Conjugation by $w_0$ leaves stable the set of simple reflections $S$ and acts on  $S$ by the automorphism $d$.


\subsection{}
\label{ss extended affine}
Let $(Y, \alpha'_i, \alpha'^\vee_i), i \in [n]$ be the finite root system specifying a reductive algebraic group G over  $\CC$.  Denote the Weyl group, simple reflections, roots, and root lattice by $W_f, S_f, R'_f, Q'_f$. The \emph{extended affine Weyl group} is the semidirect product
\[\eW :=Y \rtimes W_f.\]
Elements of $Y \subseteq \eW$ will be denoted by the multiplicative notation $y^\lambda, \lambda\in Y$.

The group $W_e$ is also equal to $\Pi \ltimes W_a$, where $W_a$ is the Weyl group of an affine root system we will now construct and $\Pi$ is an abelian group. Let $X=Y^\vee\oplus\ZZ$ and $\delta$ be a generator of $\ZZ$. The pairing of $X$ and $X^\vee$ is obtained by extending the pairing of $Y$ and $Y^\vee$ together with $\langle\delta, Y\rangle=0$. Let $\phi'$ be the dominant short root of $(Y, \alpha'_i, \alpha'^\vee_i)$ and $\theta=\phi'^\vee$ the highest coroot. For $i\neq 0$ put $\alpha_i=\alpha'^\vee_i$ and $\alpha^\vee_i=\alpha'_i$; put $\alpha_0=\delta-\theta$ and $\alpha^\vee_0=-\phi'$. Then $(X, \alpha_i, \alpha^\vee_i)$, $i \in [0,n]$ is an affine root system. Let $W_a$ denote its Weyl group and use the notation of \textsection\ref{ss root system} for its roots, root lattice, etc. The roots $R$ may be expressed in terms of the coroots $\fcoroots$ of the system $(Y, \alpha'_i, \alpha'^\vee_i)$ as $R = \fcoroots + \ZZ\delta$, and the positive roots $R_+$ as $R_+ = (\fcoroots + \ZZ_{>0}\delta) \cup \fcoroots_+$.

The abelian group $Q'_f$ is realized as a subgroup of $W_a$ acting on $X$ and $X^\vee$ by translations: for $\beta' = \alpha'_i \in R'_f \subseteq Q'_f$ $(i \in [n])$, define $y^{\beta'} = s_{\alpha'^\vee_i - \delta} s_{\alpha'^\vee_i}$. Then
\be y^{\beta'}(x) = x - \langle x, \beta' \rangle\delta, \quad x \in X, \ee
\be \label{e translation action}
y^{\beta'}(x^\vee) = x^\vee + \langle \delta, x^\vee \rangle \beta', \quad x^\vee \in X^\vee, \ee
and for any $\beta' \in Q'_f$, these equations define an action of $Q'_f$ on $X$ and $X^\vee$.
This action of $Q'_f$ by translations extends to an action of $Y$, which realizes $W_e$ as a subgroup of the automorphisms of $X$ and $X^\vee$. The inclusion $W_a\hookrightarrow W_e$ is given on simple reflections by $s_i \mapsto s_i$ for $i\neq 0$ and $s_0 \mapsto y^{\phi'}s_{\phi'}$. The subgroup $\aW$ is normal in $\eW$ with quotient $W_e/W_a \cong Y/Q'_f$, denoted $\Pi$. And, as was our goal, we have $W_e=\Pi\ltimes W_a$.

Let $H=\{x\in X^\vee_\RR:\langle\delta,x\rangle=1\}$ be the \emph{level 1 plane}. It follows from (\ref{e translation action}) that the action of $W_e$ on $X^\vee$ restricts to an action of $W_e$ on $H$. The space $H$ contains the affine hyperplanes $h_\alpha := H_\alpha \cap H, \alpha \in R$. The connected components of $H - \bigcup_{\alpha \in R} h_\alpha$ are \emph{alcoves}, and the \emph{basic alcove} is $\alcove= H \cap \chamber$. Its closure is a fundamental domain for the action of $W_a$ on $H$, and its stabilizer for the action of $W_e$ is $\Pi$. We also distinguish the \emph{finite dominant Weyl chamber} $\chamberf=\{x\in H : \langle \alpha,x \rangle > 0 \text{ for all } \alpha\in \fcoroots\}$.

A basic fact we will make frequent use of is that any element $w \in W_e$ can be written uniquely as a product
 \begin{equation}\label{3factors eq}
 w =  u \cdot y^{\beta} v,
\end{equation} such that $u, v \in W_f$, $\beta \in Y$, and $y^{\beta} v$ is minimal in its right coset $W_f w$.  This last condition implies $u \cdot (y^{\beta} v)$ is a reduced factorization and $\beta \in Y_+$.  In terms of the alcove picture, $y^{\beta} v$ takes $\alcove$ to $y^{\beta} v(\alcove) \subseteq \chamberf$ and then $u$ moves this alcove into $u(\chamberf)$.  Note that $\beta \in Y_{++}$ forces $y^{\beta} v$ to be minimal in $W_f y^\beta v$ for any $v \in W_f$.   

The group $\eW$ is an extended Coxeter group. The length function and partial order on $W$ extend to $\eW$: $\ell(\pi v) = \ell(v)$, and $\pi v \leq \pi' v'$ if and only if $\pi = \pi'$ and $v \leq v'$, where $\pi, \pi' \in \Pi$, $v, v' \in W$. The definitions of left and right descent sets and reduced factorization carry over identically.

\subsection{}
\label{ss type A}
We will give examples in type  $A$ throughout the paper, and  now fix notation for this special case.  See  \cite{B3,X2} for a more extensive treatment of this case.

For $G = GL_n$, the lattices $Y$ and $Y^\vee$ are equal to $\ZZ^n$ and $\alpha'_i=\epsilon_i-\epsilon_{i+1}$, $\alpha'^\vee_i=\epsilon^\vee_i-\epsilon^\vee_{i+1}$, where $\epsilon_i$ and $\epsilon^\vee_i$ are the standard basis vectors of $Y$ and $Y^\vee$. The finite Weyl group $W_f$ is $\S_n$ and the subgroup $\Pi$ of $W_e$ is $\ZZ$. The element $\pi=y_1s_1s_2\ldots s_{n-1}\in\Pi$ is a generator of $\Pi$. This satisfies the relation $\pi s_i=s_{i+1}\pi$, where the subscripts of the $s_i$ are taken mod $n$.

For $G = SL_n$, the lattice $Y$ is the quotient of the weight lattice for  $GL_n$ by $\ZZ\varepsilon$, where $\varepsilon = \epsilon_1 + \ldots + \epsilon_n$ and the simple roots are the images of those for $GL_n$. The dual lattice  $Y^\vee$ is the coroot lattice of  $GL_n$, and the coroots are the same as those for $GL_n$.  The finite Weyl group $W_f$ is the same as for $GL_n$ and the subgroup $\Pi = \langle\pi \rangle$ is that for  $GL_n$ with the additional relation $\pi^n = \idelm$.

Another description of $\eW$ for  $GL_n$, due to Lusztig, identifies it with the group of permutations $w: \ZZ \to \ZZ$ satisfying $w(i+n) = w(i)+n$ and $\sum_{i = 1}^n (w(i) - i) \equiv 0$ mod $n$. The identification takes $s_i$ to the permutation transposing $i+kn$ and $i+1+kn$ for all $k \in \ZZ$, and takes $\pi$ to the permutation $k \mapsto k+1$ for all $k \in \ZZ$. We take the convention of specifying the permutation of an element $w \in \eW$ by the word
\[{\small w(1)\ w(2)\dots w(n).}\]
We refer to this as the \emph{word of} $w$, also written as $w_1 w_2 \cdots w_n$; this is understood to be part of an infinite word so that $w_i = i-\rsd{i} + w_{\rsd{i}}$, where $\rsd{i}$ denotes the element of $[n]$ congruent to $i$ mod  $n$. For example, if $n = 4$ and $w = \pi^2 s_2 s_0 s_1$, then the word of $w$ is $5\ 2\ 4\ 7$.

The extended affine Weyl group for  $SL_n$ may be obtained from this permutation group by quotienting by the subgroup generated by $\pi^n = n+1 \ n+2 \cdots 2n$.

\subsection{}
Let $A = \ZZ[\u,\ui]$ be the ring of Laurent polynomials in the indeterminate $\u$ and $A^{-}$ be the subring $\ZZ[\ui]$. The \emph{Hecke algebra} $\H(W)$ of a (extended) Coxeter group $(W, S)$ is the free $A$-module with basis $\{T_w :\ w\in W\}$ and relations generated by
\begin{equation}\label{hecke eq}
\begin{array}{ll}T_uT_v = T_{uv} & \text{if } uv = u\cdot v\ \text{is a reduced factorization}\\
 (T_s - \u)(T_s + \ui) = 0 & \text{if } s\in S.\end{array}
\end{equation}

The bar-involution, $\br{\cdot}$, of $\H$ is the additive map from $\H$ to itself extending the involution $\br{\cdot}$: $A\to A$ given by $\br{q} = q^{-1}$ and satisfying $\br{T_w} = T_{w^{-1}}^{-1}$.
 Define the lattice
\[ \L = A^{-}\langle T_w : w \in W \rangle. \]
 \begin{theorem}[Kazhdan-Lusztig \cite{KL}]
 For each $w \in W$, there is  a unique element $\C_w \in \H(W)$ such that $\br{\C_w} = \C_w$ and $\C_w$ is congruent to $T_w \mod \ui \L$.
 \end{theorem}
The set $\{\C_w : w\in W\}$ is an $A$-basis for $\H(W)$ called the \emph{canonical basis} or Kazhdan-Lusztig basis.

The coefficients of the $\C$'s in terms of the $T$'s are the \emph{Kazhdan-Lusztig polynomials} $P'_{x,w}$:
\be \C_w = \sum_{x \in W} P'_{x,w} T_x. \ee
(Our $P'_{x,w}$ are equal to $q^{(\ell(x)-\ell(w))/2}P_{x,w}$, where $P_{x,w}$ are the polynomials defined in \cite{KL}.)

\subsection{}
The \emph{extended affine Hecke algebra} $\eH$ is the Hecke algebra $\H(\eW)$. Just as the extended affine Weyl group $W_e$ can be realized both as $\Pi\ltimes W_a$ and $W_f\ltimes Y$, the extended affine Hecke algebra can be realized in two analogous ways:

The algebra $\eH$ contains the Hecke algebra $\H(\aW)$ and is isomorphic to the twisted group algebra $\Pi\cdot\H(W_a)$ generated by $\Pi$ and $\H(W_a)$ with relations generated by
$$\pi T_w=T_{\pi w \pi^{-1}}\pi$$
for $\pi\in\Pi$, $w\in W_a$.

There is also a presentation of $\eH$ due to Bernstein. For any $\lambda\in Y$ there exist $\mu,\nu\in Y_+$ such that $\lambda = \mu-\nu$. Define
\[\y^\lambda := T_{y^\mu}(T_{y^\nu})^{-1},\]
which is independent of the choice of $\mu$ and $\nu$. The algebra $\eH$ is the free $A$-module with basis $\{\y^\lambda T_w :\ w\in W_f,\lambda\in Y\}$ and relations generated by
\[ \begin{array}{ll}T_i\y^\lambda=\y^\lambda T_i & \text{if }\langle\lambda,\alpha'^\vee_i\rangle=0,\\
T^{-1}_i\y^\lambda T^{-1}_i=\y^{s_i(\lambda)} & \text{if }\langle \lambda,\alpha'^\vee_i\rangle=1,\\
(T_i - \u)(T_i + \ui) = 0 & \end{array}\]
for all $i \in [n]$, where $T_i := T_{s_i}$.

\subsection{}
Given $\lambda \in Y_+$, let $\chi_\lambda(Y) = \sum_\mu d_{\mu,\lambda} Y^\mu$, where $d_{\mu,\lambda}$ is the dimension of the $\mu$-weight space of the irreducible representation of $\lie(G)$ of highest weight $\lambda$.

\begin{theorem}[Lusztig {\cite[Proposition 8.6]{L}}]
\label{t Lusztig}
For $\lambda \in Y_+$, the canonical basis element $\C_{w_0 y^\lambda}$ can be expressed in terms of the Bernstein generators as
\[\C_{w_0 y^\lambda}=\chi_\lambda(\y)\C_{w_0}=\C_{w_0}\chi_\lambda(\y).\]
\end{theorem}


\section{Primitive elements}
\label{section primitive}
\subsection{}
The primitive elements of $W_e$ that we are about to define are most natural in the case $G$ is simply connected, so let us assume this from now on. This is equivalent to the assumption that \emph{fundamental weights} $\varpi_i$, $i \in [n]$ exist and are $\ZZ$-basis for $Y$. (The weight $\varpi_i$ is defined by $\langle \varpi_i, \alpha'^\vee_j \rangle$ equals 1 if $i = j$ and 0 otherwise.)

We give three descriptions of primitive elements and show that they are equivalent in Proposition \ref{primitive proposition}. The first description is a geometric one from \cite{L Jantzen}. A \emph{box} is a connected component of $H - \bigcup_{i \in [n], k \in \ZZ} h_{\alpha_i + k\delta}$. We denote by $\boxo$ the box containing $\alcove$. It is bounded by the hyperplanes $h_{\alpha_i}$ and $h_{\alpha_i - \delta}$ for $i \in [n]$.

The action of $Y$ on $H$ by translations gives the action $y^\lambda(h_\alpha) = h_{\alpha-\langle \alpha,\lambda \rangle\delta}$ of $Y$ on hyperplanes. This further gives an action of $Y$ on boxes. Put $\lambda = \sum_{i = 1}^n c_i\varpi_i$, $c_i \in \ZZ$, and define $\bx{\lambda} = y^\lambda(\boxo)$. It is the box that is the bounded component of
\[H - \bigcup_{i} h_{\alpha_i - c_i\delta} - \bigcup_i h_{\alpha_i - (c_i+1)\delta}.\]
Thus our assumption that the fundamental weights are a basis for $Y$ implies that $Y$ acts simply transitively on boxes. Additionally, the $\bx{\lambda}$ for $\lambda \in Y^+$ are the connected components of $\chamberf - \bigcup_{i \in [n],k \in \ZZ} h_{\alpha_i + k\delta}$.

Set $\rho = \sum_{i = 1}^n \varpi_i$. One checks that $w_0(\boxo) = \bx{-\rho}$.

Given any $v\in W_f$, let $J=R(v)$. The element $v^{-1}$ takes the basic alcove $\alcove$ to some alcove $v^{-1}(\alcove)$ whose closure contains the origin. There is a unique minimal $\lambda\in Y_+$ such that $y^{-\lambda} v^{-1}(\alcove) \subseteq w_0(\chamberf)$. Minimality implies that $y^{-\lambda} v^{-1}(\alcove) \subseteq \bx{-\rho}$. It is a consequence of the next proposition that this $\lambda$ is given by
\[ \lambda = \sum_{s_i\in S_f\backslash J} \varpi_i. \]
Now define $w$ to be
\[ v \cdot y^\lambda= y^{v(\lambda)} v, \]
which is maximal in its left coset $wW_f$.
For example, for $G = SL_5$ (see \textsection\ref{ss type A}), if $v=5\ 2\ 3\ 1\ 4$, then $J=\{s_1,s_3\}$, $\lambda=(2,2,1,1,0)$, and $v(\lambda) = (1,2,1,0,2)$.


\begin{proposition}
\label{primitive proposition}
The following are equivalent for an element $w \in W_e$:
\setcounter{ctr}{0}
\begin{list}{\emph{(\roman{ctr})}} {\usecounter{ctr} \setlength{\itemsep}{1pt} \setlength{\topsep}{2pt}}
\item $w^{-1}(\alcove) \subseteq \boxo$.
\item $w(\alpha_i) \in (R_- +\delta) \cap R_+ = (\fcoroots_-+\delta) \cup \fcoroots_+$ for $i \in [n]$.
\item $w = vy^{\lambda}w_0$ such that $v \in W_f$, and $\lambda=\sum_{s_i\in S_f\backslash J} \varpi_i$, where $J=R(v)$ as above.
\end{list}
\end{proposition}

\begin{proof}
The equivalence of (i) and (ii) is proved by observing that each of the statements below is equivalent to the next. The equivalence of (a) and (b) follows from (\ref{e hyperplane separate}).
\begin{list}{\emph{(\alph{ctr})}} {\usecounter{ctr} \setlength{\itemsep}{1pt} \setlength{\topsep}{2pt}}
\item $\alcove$ and $w^{-1}(\alcove)$ are on the same side of $h_{\alpha_i + k\delta}$ for all $i \in [n], k \in \ZZ$.
\item ($w(\alpha_i+k\delta) \in R_+$ and $\alpha_i + k\delta \in R_+$) or ($w(\alpha_i + k\delta) \in R_-$ and $\alpha_i + k\delta \in R_-$) for all $i \in [n], k \in \ZZ$.
\item ($w(\alpha_i)+k\delta \in R_+$ and $k \geq 0$) or ($w(\alpha_i) + k\delta \in R_-$ and $k < 0$) for all $i \in [n], k \in \ZZ$.
\item $w(\alpha_i) \in R_+$ and $w(\alpha_i) - \delta \in R_-$ for all $i \in [n]$.
\item $w(\alpha_i) \in (R_- + \delta) \cap R_+$ for all $i \in [n]$.
\end{list}

To see (iii) implies (ii), for any $i\in [n]$, compute
\[ vy^{\lambda}w_0(\alpha_i) = vy^{\lambda}(-\alpha_j)=v(-\alpha_j-\langle\lambda,-\alpha_j\rangle\delta)=v(-\alpha_j)+
\left\{\begin{array}{ll}
\delta & \text{if }j\not\in J,\\
0 & \text{if }j\in J
\end{array},\right. \]
where $j = d(i)$ (with $d$ as in \textsection\ref{ss w0} so that $w_0(\alpha_i) = -\alpha_j$). The condition $v(-\alpha_j)\in \fcoroots_+$ is equivalent to $j\in R(v) = J$, hence $vy^{\lambda}w_0(\alpha_i)\in (\fcoroots_-+\delta) \cup \fcoroots_+$.

Now assume $w$ satisfies (ii). Put $J = \{s_j : ww_0(-\alpha_j)\in \fcoroots_+\}$ and define $\lambda := \sum_{s_i\in S_f\backslash J} \varpi_i$. Then, define $v := ww_0y^{-\lambda}$ and compute
\[ ww_0y^{-\lambda}(-\alpha_j) = ww_0(-\alpha_j-\langle-\lambda,-\alpha_j\rangle\delta) = ww_0(-\alpha_j)-\left\{
\begin{array}{ll}
\delta & \text{if }j\not\in J,\\
0 & \text{if }j\in J
\end{array}
\right. \]
where, as above, $i = d(j)$. By the assumption (ii) and definition of $J$, $v(-\alpha_j)\in \fcoroots$ for $j\in [n]$. Writing $v = uy^\mu$, $u \in W_f$, $\mu \in Y$, and using that the fundamental weights form a basis for $Y$, we may conclude that $\mu = 0$, i.e., $v \in W_f$. Also, $R(v) = \{s_j : v(-\alpha_j)\in \fcoroots_+\} = J$, so $w= vy^{\lambda}w_0$ with the conditions in (iii) satisfied.
\end{proof}

\begin{definition}
A $w \in W_e$ satisfying any (all) of the preceding conditions is called \emph{primitive}.
\end{definition}


\begin{proposition}
For $G = SL_n$, $x \in W_e$ is primitive if and only if  $\ 1 \leq x_{i+1}-x_i\leq n$ for $i=1,\ldots,n-1$, where $x_1\ x_2\dots x_{n-1} \ x_n$ is the word of $x$ (see \textsection\ref{ss type A}).
\end{proposition}

\begin{proof}
The word of $x$ and $x$ as an automorphism of $X$ are related by
\[x(\epsilon^\vee_i) = \epsilon^\vee_{\rsd{x}_i}+\left(\frac{\rsd{x}_i-x_i}{n}\right)\delta.\]
Define  a function $\tau: R\to \ZZ$ by $\alpha \mapsto\langle\alpha,\rho+n\Lambda^{\vee}\rangle$, where  $\Lambda^\vee$ is the generator of  $\ZZ$ in  $X^\vee = Y \oplus \ZZ$ satisfying  $\langle\delta,\Lambda^\vee\rangle =1$. This function takes $\alpha_i$ to $1$ for $i\in [n]$. The inverse image of $[n]$ under this map is $(\fcoroots_-+\delta) \cup \fcoroots_+$. Then \[\tau(x(\epsilon^\vee_i)) = n-\rsd{x}_i + \frac{\rsd{x}_i - x_i}{n}\ n = n-x_i,\] so $\tau(x(\alpha_i)) = x_{i+1} - x_i$ is in $[n]$ if and only if $x(\alpha_i)\in (\fcoroots_-+\delta) \cup \fcoroots_+$.
\end{proof}

\begin{example}
For $G=SL_4$, the primitive elements of $W_e$, expressed as products of simple reflections (top lines) and words (bottom lines), are
{\linespread{1.8}
\Large
\[\begin{array}{c}
\substack{\idelm\\1\ 2\ 3\ 4}\\
\substack{\pi s_1s_0\\\ 1\ 2\ 4\ 7} \quad \substack{\pi s_0\\1\ 3\ 4\ 6} \quad \substack{\pi\\2\ 3\ 4\ 5}\\
\substack{\pi^2s_1s_3s_0\\1\ 3\ 6\ 8} \quad \substack{\pi^2s_3s_0\\1\ 3\ 5\ 7} \quad \substack{\pi^2s_1s_0\\2\ 3\ 5\ 8} \quad \substack{\pi^2s_0\\2\ 4\ 5\ 7} \quad \substack{\pi^2\\3\ 4\ 5\ 6}\\
\substack{\pi^3s_2s_1s_3s_0\\1\ 4\ 7\ 10} \quad \substack{\pi^3s_1s_3s_0\\2\ 4\ 7\ 9} \quad \substack{\pi^3s_1s_0\\3\ 4\ 6\ 9} \quad \substack{\pi^3s_3s_0\\2\ 5\ 7\ 8} \quad \substack{\pi^3s_0\\3\ 5\ 6\ 8} \quad \substack{\pi^3\\4\ 5\ 6\ 7}\\
\substack{\pi^4s_1s_3s_0\\3\ 5\ 8\ 10} \quad \substack{\pi^4s_2s_1s_3s_0\\2\ 6\ 8\ 11} \quad \substack{\pi^4s_3s_0\\3\ 6\ 8\ 9} \quad \substack{\pi^4s_1s_0\\4\ 5\ 7\ 10} \quad \substack{\pi^4s_0\\4\ 6\ 7\ 9}\\
\substack{\pi^5s_2s_1s_3s_0\\3\ 6\ 9\ 12} \quad \substack{\pi^5s_1s_3s_0\\4\ 6\ 9\ 11} \quad \substack{\pi^5s_3s_0\\4\ 7\ 9\ 10}\\
\substack{\pi^6s_2s_1s_3s_0\\4\ 7\ 10\ 13}
\end{array}\]}
\end{example}

\subsection{}
Let us now establish some properties of primitive elements. The first two are basic and the third is more substantial.
\begin{proposition}\label{pi mult prop}
For any $\pi\in\Pi$, $w\in W_e$ is primitive if and only if $\pi w$ is.
\end{proposition}

\begin{proof}
Given any $i\in[n]$, $w$ primitive implies $w(\alpha_i)\in R_+$ and $w(\alpha_i-\delta)\in R_-$.  Since $\pi(R_+)=R_+$ (and $\pi(R_-)=R_-$), $\pi w(\alpha_i)\in R_+$ and $\pi w(\alpha_i-\delta)\in R_-$.  In other words, $\pi w(\alpha_i) \in (R_- +\delta) \cap R_+$, so $\pi w$ is primitive. The same argument with $\pi^{-1}$ in place of $\pi$ and $\pi w$ in place of $w$ gives the other direction.
\end{proof}

\begin{proposition}
\label{partial primitive prop}
If $x\cdot y$ is a reduced factorization of a primitive element, then $y$ is primitive.
\end{proposition}

\begin{proof}
For $i \in [n]$, $y(\alpha_i)\in R_{+}$ because if not, $y(\alpha_i)\in R_{-}$ and $x\cdot y$ reduced implies by Lemma \ref{reducedlemma} that $xy(\alpha_i)\in R_-$, contradiction. Also, $x\cdot y$ reduced, $\alpha_i-\delta\in R_-$, and $xy(\alpha_i-\delta)\in R_-$ implies $y(\alpha_i-\delta)\in R_-$. Thus $y(\alpha_i)\in (R_- +\delta) \cap R_+$ shows $y$ is primitive.
\end{proof}

The \emph{lowest two-sided cell} of $W_e$ is the set $W_{(\nu)} = \{w \in W_e: w = x \cdot w_0 \cdot z, \text{ for some } x,z \in W_e\}$ (see \cite{X, X2}).

\begin{proposition}
\label{p two-sided primitive}
For any $w \in W_{(\nu)}$, there exists a unique expression for $w$ of the form
\be w = v_1 \cdot w_0 y^\lambda \cdot v_2 \ee
where $v_1, v_2^{-1}$ are primitive and $\lambda \in Y_+$.
\end{proposition}

\begin{proof}
Write $w = x \cdot w_0 \cdot z$. By (\ref{3factors eq}), $z(\alcove) \subseteq \chamberf$ so $z(\alcove) \subseteq \bx{\eta}$ for some $\eta \in Y_+$. Thus by Proposition \ref{primitive proposition} and the discussion preceding it, $z = y^\eta v_2$ for some primitive element $v_2^{-1}$. Similarly, $x = v_1 y^{-\mu}$ for $v_1$ primitive and $\mu \in Y_+$. Next, we can write $y^{-\mu} w_0 y^\eta = w_0 y^{w_0(-\mu)+\eta}$. Setting $\lambda = w_0(-\mu)+\eta$ and noting that $w_0(-\mu) \in Y_+$ yields the desired expression $w = v_1 w_0 y^\lambda v_2$, with $v_1, v_2^{-1}$ primitive. To see that this factorization is reduced, use that $v_1 \cdot w_0$ and $w_0 \cdot y^\lambda v_2$ are reduced and Proposition \ref{key_factorization_proposition} (see \textsection\ref{ss factorthm 1} below) to conclude that $v_1 \cdot w_0 y^\lambda v_2$ is reduced. Similarly, by rewriting $w_0 y^\lambda = y^{w_0(\lambda)}w_0$ we may conclude $v_1 y^{w_0(\lambda)} w_0 \cdot v_2$ is reduced.

For uniqueness, suppose that $v_1 \cdot w_0y^\lambda \cdot v_2 = w = v'_1 \cdot w_0y^{\lambda'} \cdot v'_2$ for $v_1, v'_1, v_2^{-1}, v_2'^{-1}$ primitive and $\lambda, \lambda' \in Y_+$. Put $v_1 w_0 = u_1 y^\mu$, $v_1' w_0 = u_1' y^{\mu'}$ with $u_1, u_1' \in W_f$, $\mu, \mu' \in Y^+$ as in Proposition \ref{primitive proposition} (iii). Then $w(\alcove) = u_1 y^{\mu+\lambda} v_2(\alcove) \subseteq u_1(\chamberf)$ and also $u_1'(\alcove) = u_1'y^{\mu' + \lambda'} v_2'(\alcove) \subseteq u_1'(\chamberf)$. Thus $u_1 = u_1'$, implying $v_1 = v_1'$. Similarly, $v_2 = v_2'$ and then $\lambda = \lambda'$ follows easily.
\end{proof}

\section{The factorization theorem}\label{section Fact Theorem}
\newcommand{\lC}{\ensuremath{\overleftarrow{C}^{\prime}}}
\newcommand{\rC}{\ensuremath{\overrightarrow{C}^\prime}}
\newcommand{\lP}{\ensuremath{\overleftarrow{P}}}
\newcommand{\rP}{\ensuremath{\overrightarrow{P}}}
\newcommand{\lzy}{y^{\lambda}}
\subsection{}
\label{ss factorthm 1}
This section is devoted to a proof of our main theorem, which we now state. Suppose $v \in W_e$ such that $v \cdot w_0$ is reduced. It is well-known that
\[ \C_{v w_0} = \sum_x \lP_{x,v} T_x \C_{w_0}, \]
where the sum is over $x \leq v$ such that $x \cdot w_0$ is reduced and $\lP_{x,v} := P'_{x w_0, v w_0}$ (see \cite{B0} for the more general construction of which this is a special case).
Define $\lC_v = \sum_x \lP_{x,v} T_x$, with the sum over the same $x$ as above. Similarly, for $v$ such that $w_0 \cdot v$ is reduced, define $\rP_{x,v} = P'_{w_0x,w_0v}$ and $\rC_{v} = \sum_x \rP_{x,v} T_x$.

\begin{theorem}\label{main theorem}
Let $v_1,\; v_2 \in W_e$ such that $v_1 \cdot w_0$ and $w_0 \cdot v_2$ are reduced factorizations.  If $v_1$ is primitive, then $\C_{v_1 w_0 v_2} = \lC_{v_1} \C_{w_0} \rC_{v_2} $.
\end{theorem}

This theorem together with Theorem \ref{t Lusztig} and Proposition \ref{p two-sided primitive} have the following powerful corollary for any $w$ in the lowest two-sided cell of $W_e$. This is phrased a little differently from the result stated in the abstract, but is equivalent to it.
\begin{corollary}\label{main corollary}
For $w \in W_{(\nu)}$ and with  $w = v_1 \cdot w_0 y^\lambda \cdot v_2$ as in Proposition \ref{p two-sided primitive}, we have the factorization
\be \C_{w} =  \chi_\lambda({\y}) \lC_{v_1} \C_{w_0} \rC_{v_2} .\ee
\end{corollary}


We will need the two general lemmas about Coxeter groups that follow.  Their proofs are straightforward and are given in the appendix.

\begin{lemma} \label{reducedlemma}
For any $x, y \in W_e$, $x \cdot y$ is a reduced factorization if and only if $x(y(R_{+}) \cap R_{-}) \subseteq R_{-}$.
\end{lemma}

The next lemma holds for any Weyl group, but it is stated in the less general setting in which it will be applied.
\begin{lemma}\label{long z helper 2}
Suppose $a,y\in W_f$ and $\alpha\in \fcoroots_+$.
\setcounter{ctr}{0}
\begin{list}{\emph{(\roman{ctr})}} {\usecounter{ctr} \setlength{\itemsep}{1pt} \setlength{\topsep}{2pt}}
\item If $s_{\alpha}a < a$, then $s_{\alpha}ay < ay \iff a^{-1}s_{\alpha}ay > y$.
\item If $s_{\alpha}a > a$, then $s_{\alpha}ay > ay \iff a^{-1}s_{\alpha}ay > y$.
\end{list}
\end{lemma}

The next proposition is the crux of the proof of Theorem \ref{main theorem} and is the only place the primitive assumption is used directly.
\begin{proposition}\label{key_factorization_proposition}
Let $x,\; z \in W_e$ such that $x \cdot w_0$ and $w_0 \cdot z$ are reduced factorizations.  Then $x \cdot w_0 \cdot z$ is a reduced factorization.  Furthermore, if $x$ is primitive, then $x \cdot y \cdot z$ is a reduced factorization for every $y \in W_f$ with $\ell(y) = \ell(w_0)-1$.
\end{proposition}
\begin{proof}
Given $\alpha \in R_{+}$, suppose $z(\alpha) \in R_-$.  We will show that $x w_0 z(\alpha)$ remains in $R_-$, which by Lemma \ref{reducedlemma} shows $x w_0 \cdot z$ is a reduced factorization.  Write $z(\alpha) = \alpha' - k \delta$, for some $\alpha' \in \fcoroots$ and $k \geq 0$.  In fact, $\alpha' \in \fcoroots_+$ since otherwise $z(\alpha + k\delta) \in \fcoroots_- \subseteq R_-$ and $w_0 z(\alpha + k\delta) \in \fcoroots_+ \subseteq R_+$, contradicting $w_0 \cdot z$ reduced by Lemma \ref{reducedlemma}.  Now
$$x w_0 z(\alpha) = x w_0 (\alpha' - k \delta) = x(w_0(\alpha')) - k \delta.$$
The root $x(w_0(\alpha'))$ is in $R_-$ because $x \cdot w_0$ is reduced and $\alpha' \in R_+$ while $w_0(\alpha') \in R_-$.  Therefore $x(w_0(\alpha')) - k \delta$ is in $R_-$, as desired.

Now suppose $x$ is primitive and $y \in W_f$ with $\ell(y) = \ell(w_0)-1$.  By \textsection \ref{ss w0}, we have $y = w_0 s_i$ for some $i \in [n]$.  Certainly $x \cdot y$ is reduced since a reduced factorization for $xy$ in terms of simple reflections can be obtained from one for $x w_0$ that ends in $s_i$, by deleting that last $s_i$. Proceed as in the first part of the proof by considering any $\alpha \in R_{+}$ such that $z(\alpha) \in R_-$, and showing that $x y z(\alpha)$ remains in $R_-$.  As above, $z(\alpha) = \alpha' - k \delta \in R_-$ with $\alpha' \in \fcoroots_+$, and hence $k \geq 1$.  We may assume $\alpha' = \alpha_i$ because if not, the same argument used above works. Then
$$x y z(\alpha) = x(y(\alpha')) - k \delta = x(\alpha_j) - k \delta,$$
where $j \in [n]$ so that $\alpha_j = w_0(-\alpha_i)$ (see \textsection\ref{ss w0}).  Now $x$  primitive implies $x(\alpha_j) \in (\fcoroots_-+\delta) \cup \fcoroots_+$ and thus $xyz(\alpha) = x(\alpha_j) - k \delta \in R_-$,  as $k \geq 1$.
\end{proof}

\subsection{}
Here we prove two technical lemmas whose significance may not become clear until seeing their application in the main thread of the proof in \textsection\ref{main proof section}.

\newcommand{\sz}{\ensuremath{z}}
\newcommand{\lz}{\ensuremath{\tilde{z}}}

Given $x \in W_a$ and an element $\lz = u y^{\beta} v$ factored as in (\ref{3factors eq}), we will say $\lz$ is \emph{large with respect to $x$} if any (all) of the following equivalent conditions is satisfied:
\be
\label{e zlarge}
\begin{array}{rl}
\text{(i)}&\langle \beta, \alpha_i \rangle \  >> \ell(x) + \ell(w_0)\text{ for }i \in [n]. \\
\text{(ii)}& \lz(\alcove) \text{ is far from the affine hyperplanes } h_\alpha \text{ (i.e. }\langle \alpha, \lz(\alcove) \rangle \\&>> \ell(x) + \ell(w_0)\text{) for all } \alpha \in \fcoroots.
\end{array}
\ee

The next lemma is the only place the largeness assumption is used directly.


Before stating the lemma, define the homomorphism
\be \Psi: W_a \cong W_f \ltimes Q'_f \twoheadrightarrow W_f \ee
to be the projection onto the first factor. Geometrically, this map can be understood by the action of $W_a$ on $X^\vee$. This action leaves stable the level 0 plane $\{x\in X^\vee_\RR:\langle\delta,x\rangle=0\} \cong Y$, and $\Psi$ is given by restricting the action of $W_a$ to this plane. On simple reflections, $\Psi$ is given by
\be
\Psi(s_i) =
\left\{ \begin{array}{ll}
s_i & \text{if $i \in [n]$}, \\
s_{\phi'} & \text{if } i = 0.
\end{array} \right.
\ee

\begin{lemma}\label{long z helper 1} If $x \in W_a$, $\lz = u y^{\beta} v$ as in (\ref{3factors eq}), and $\lz$ large with respect to $x$, then
\setcounter{ctr}{0}
\begin{list}{(\roman{ctr})} {\usecounter{ctr} \setlength{\itemsep}{1pt} \setlength{\topsep}{2pt}}
\item if $x \lz = u' y^{\beta'} v$ is the unique factorization of (\ref{3factors eq}), then $u' = \Psi(x) u \in W_f$, i.e., $\lj{(x\lz)}{S_f} = \Psi(x)\lj{\lz}{S_f} = \Psi(x)u$,
\item $s_0 \lz < \lz$ if and only if $\Psi(s_0) \lz > \lz$,
\item if $u = \idelm$, then $s_0 x \lz < x \lz \iff \Psi(s_0 x) > \Psi(x)$.  Similarly, if $u =\idelm$ and $i \in [n]$, then $s_i x \lz < x \lz \iff \Psi(s_i x) < \Psi(x)$.
\end{list}
\end{lemma}
\begin{proof}
  Part (i) of the lemma follows easily from the case $x = s_0$.  In that case,
  $$ s_0 \lz = y^{\phi'} s_{\phi'} u \ y^{\beta} v =  s_{\phi'} u (u^{-1} s_{\phi'} y^{\phi'} s_{\phi'} u) \ y^{\beta} v =  \Psi(s_0) u \ y^{u^{-1}s_{\phi'}({\phi'})+\beta} v.$$
  Define $\beta' = u^{-1}s_{\phi'}({\phi'})+\beta$.  Since $\langle \beta, \alpha_i \rangle >> 0$, the same holds for $\beta'$ because $\langle u^{-1}s_{\phi'}({\phi'}), \alpha_i \rangle$ is bounded by a constant depending only on $W_f$.  As mentioned in \textsection\ref{ss extended affine}, $\beta'\in Y_{++}$ implies that $\Psi(s_0) u \ y^{\beta'} v$ is the desired factorization for $s_0 \lz$. That $\lz$ is large with respect to $x$ ensures that we can multiply $\lz$ on the left by $\ell(x)$ simple reflections to obtain $u'y^{\beta'} v$ and still have $\beta' \in Y_{++}$.


For statement (ii) we have the equivalences
\be
\begin{array}{ccc}
s_0\lz < \lz &\iff  &\\
\alcove \text{ and } \lz(\alcove) \text{ are on opposite sides of } h_{\alpha_0} & \iff & \\
\alcove \text{ and } \lz(\alcove) \text{ are on the same side of } h_\theta & \iff & \\
\Psi(s_0) \lz > \lz
\end{array}
\ee
where the first and the third equivalences follow from (\ref{e hyperplane separate}) and the second from (\ref{e zlarge}).

By applying statement (ii) to $x \lz$, we have $s_0 x \lz < x \lz$ if and only if $\Psi(s_0)  x \lz > x \lz$.  Next use that the factorizations $x \lz = u' \cdot y^{\beta'} v$ and $\Psi(s_0) x \lz= \Psi(s_0) u' \cdot y^{\beta'} v$ are reduced to conclude  the left-hand equality of
\be \ell(\Psi(s_0)  x \lz) - \ell(x \lz) = \ell(\Psi(s_0) u') - \ell(u')\ = \ell(\Psi(s_0x)) - \ell(\Psi(x)). \ee
The right-hand equality is the substitution $u' = \Psi(x)u = \Psi(x)$, which uses statement (i) and the assumption $u = \idelm$. Therefore, $\Psi(s_0) x \lz > x \lz$ if and only if  $\Psi(s_0x) > \Psi(x)$ giving the first part of statement (iii). The statement for $i \in [n]$ has the same proof except with the statement (ii) reference replaced by the triviality $s_i x \lz < x \lz \iff \Psi(s_i)  x \lz < x \lz$.
\end{proof}

\begin{example}Lemma \ref{long z helper 1} (iii) is fairly intuitive in type $A$. For example,  for $G = SL_5$, let $\lz^{\ng1} = \ng27\ \ng13\ 4\ 16\ 35$ (with the convention of \textsection\ref{ss type A}), and $x^{\ng1}$ be any of the six possibilities shown below.
\[
\begin{array}{rcccccccccccccc}
\multicolumn{1}{c}{x^{-1}} && \multicolumn{5}{c}{\lz^{-1}x^{-1}} &&& \multicolumn{5}{c}{\Psi(x^{-1})} \\
\idelm && \ng27&\ng13&4&16&35&&\ &1&2&3&4&5&\\
s_0 &&30&\ng13&4&16&\ng22&< &\ &5&2&3&4&1&>\\
s_0s_4&& 30&\ng13&4&\ng22&16&<&\ &5&2&3&1&4&<\\
s_0s_4s_1 &&\ng13&30&4&\ng22&16&<&\ &2&5&3&1&4&<\\
s_0s_4s_1s_0 && 11&30&4&\ng22&\ng8&<&\ &4&5&3&1&2&>\\
s_0s_4s_1s_0s_1 && 30& 11&4&\ng22&\ng8&>&\ &5&4&3&1&2&>
\end{array}
\]
The third (resp. fifth) column compares the length of an element $\lz^{-1}x^{-1}$ (resp. $\Psi(x^{-1})$) to the element immediately above it. To make these length comparisons, we use the fact that $ws_i > w$ if and only if $w_i < w_{i+1}$, where $w_1 \dots w_n$ is the word of $w$ (see \cite{X2}). Lemma \ref{long z helper 1} (iii) says that the inequalities in the third and fifth columns will match exactly when the value for $x^{-1}$ differs from the value above it by right-multiplication by some $s \in S_f$.
\end{example}

Let $x \in W_a$, $\lz \in W_e$ such that $x \cdot w_0$ and $w_0 \cdot \lz$ are reduced.  Choose reduced factorizations $x = s_{i_1} s_{i_2} \ldots s_{i_l}$, $\lz = s_{j_1} s_{j_2} \ldots s_{j_k}$.  Fix $y \in W_f$ and note that $x \cdot y$ and $y \cdot z$ are also reduced factorizations.  Suppose $s_{i_1} \in L(s_{i_2} \ldots s_{i_l} y \lz)$.  Then by the Strong Exchange Condition (see, e.g., \cite[\textsection5.8]{Hu}) there is an $r\in [k]$ such that
$$x y \lz = s_{i_1} s_{i_2} \ldots s_{i_l} y s_{j_1} s_{j_2} \ldots s_{j_k} = s_{i_2} \ldots s_{i_l} y s_{j_1} s_{j_2} \ldots \hat{s}_{j_r} \ldots s_{j_k}.$$
The Strong Exchange Condition only says that if $s\in L(w)$ (in this case, $w = s_{i_1}xy\lz$, $s = s_{i_1}$), then in any expression for $w$ as a product of simple reflections, $sw$ can be obtained by omitting one.  In this case however, $x \cdot y$ reduced implies the omitted reflection must occur in the expression for $\lz$.  Define $\lz'$ by $s_{j_1} s_{j_2} \ldots \hat{s}_{j_r} \ldots s_{j_k} = u\cdot\lz'$, where $u\in W_f$, $\lz'$ minimal in $W_f\lz'$. Also put $x' = s_{i_2} \ldots s_{i_l}$ and $y' = yu$. We then have
\begin{equation}\label{long_z_lemma eq}
xyz =   s_{i_1} s_{i_2} \ldots s_{i_l} y \lz  = s_{i_2} \ldots s_{i_l} y' \lz' = x'y'\lz'.
\end{equation}
We now can state a tricky lemma.

\begin{lemma}\label{long_z_lemma}
  With the notation above, if $\lz$ (and therefore $\lz'$) are large with respect to $x$, then $y'>y$.
\end{lemma}
\begin{example}
It may be helpful to follow the proof with an example. Let $\lz^{\ng1} = \ng27\ \ng13\ 4\ 16\ 35$, $y^{\ng1} = 4\ 3\ 1\ 2\ 5$, and $x^{\ng1}$ be the six possibilities shown below.
\[
\begin{array}{ccccccccccccccccccc}
\multicolumn{1}{c}{x^{-1}} &  \multicolumn{5}{c}{z^{-1}y^{-1}x^{-1}} &&& \multicolumn{5}{c}{\Psi(y^{-1}x^{-1})} & &\multicolumn{5}{c}{y'^{-1}}\\
\idelm & 16& 4& \ng27& \ng13& 35 & &\ & 4& 3& 1& 2& 5 &&&&&\\
s_0 &30& 4& \ng27& \ng13& 21 &<&\ & 5& 3& 1& 2& 4 &\ & 5 & 3 & 1 & 2 & 4\\
s_0s_4 &30& 4& \ng27& 21& \ng13 &>&\ & 5& 3& 1& 4& 2 & \\
s_0s_4s_1 & 4& 30& \ng27& 21& \ng13&< &\ & 3& 5& 1& 4& 2 &\ & 4& 5& 1& 2& 3\\
s_0s_4s_1s_0 & \ng18& 30& \ng27& 21& 9 &> &\ & 2& 5& 1& 4& 3 &\\
s_0s_4s_1s_0s_2 & \ng18& \ng27& 30& 21& 9&< &\  & 2& 1& 5& 4 & 3&\  & 4& 3& 5& 2& 1
\end{array}
\]
The third column indicates whether this value for $z^{-1}y^{-1}x^{-1}$ is less than or greater than the value immediately above it. The values for $y'^{-1}$ may be computed by $y'^{-1} =  \Psi(y^{-1}x^{-1}x')$ (see the proof below). According to Lemma \ref{long_z_lemma}, we must have  $5\ 3\ 1\ 2\ 4 > 4\ 3\ 1\ 2\ 5$ as $s_0 \in L(y\lz)$, $4\ 5\ 1\ 2\ 3 > 4\ 3\ 1\ 2\ 5$ as $s_1\in L(s_4s_0y\lz)$, and $4\ 3\ 5\ 2\ 1 > 4\ 3\ 1\ 2\ 5$ as $s_2\in L(s_0s_1s_4s_0y\lz)$.
\end{example}
\begin{proof}[Proof of Lemma \ref{long_z_lemma}]
Since $w_0\cdot \lz$ and $w_0\cdot \lz'$ are reduced, $\Psi(xy)=\Psi(x'y')$ by Lemma \ref{long z helper 1}(i) applied to the left- (resp. right-) hand side of (\ref{long_z_lemma eq}) with $xy$ (resp. $x'y)$ for $x$ of the lemma.  Put $a=\Psi(x')\in W_f$. Then
\be y'=a^{-1}\Psi(s_{i_{1}}) a y. \ee

First suppose $i_1 \neq 0$. Then $s_{i_1}\in L(x'y\lz)$ implies $s_{i_1} = \Psi(s_{i_1}) \in L(ay)$ by Lemma \ref{long z helper 1}(iii). On the other hand, $x\cdot w_0\cdot \lz$ and $(s_{i_{1}}x)\cdot w_0\cdot \lz$ are reduced so $s_{i_1}\notin L(x'w_0 \lz)$ implies $\Psi(s_{i_1}) \notin L(aw_0)$ again by Lemma \ref{long z helper 1}(iii). By \textsection\ref{ss w0} this implies $s_{i_1}\in L(a)$ and applying Lemma \ref{long z helper 2}(i) with $\alpha = \alpha_{i_1}$ yields the desired result.

If $i_1=0$, then $s_{i_1}\in L(x'y\lz)$ implies $\Psi(s_{i_1}) ay > ay$ by Lemma \ref{long z helper 1}(iii). On the other hand, $x\cdot w_0\cdot \lz$ and $(s_{i_{1}}x)\cdot w_0\cdot \lz$ are reduced so $s_{i_1}\notin L(x'w_0 \lz)$ implies $\Psi(s_{i_1}) a w_0 < aw_0$ again by Lemma \ref{long z helper 1}(iii). By \textsection\ref{ss w0} this implies $\Psi(s_{i_1}) a > a$ and applying Lemma \ref{long z helper 2}(ii) with $\alpha = \theta$ yields the desired result.
\end{proof}

\subsection{}
We need a basic lemma about multiplying $T$'s before giving the proof of Theorem \ref{main theorem}. For $w_1, w_2 \in W_e$, define the structure coefficients $f_{w_1, w_2, w_3} \in A$ by
\be T_{w_1} T_{w_2} = \sum_{w_3 \in W_e} f_{w_1, w_2, w_3} T_{w_3}. \ee
Let $\xi$ be the element $\u - \ui \in A$.
\begin{lemma}
\label{l structure coefficients}
The coefficients $f_{w_1, w_2, w_3}$ are polynomials in $\xi$ with non-negative integer coefficients.
\end{lemma}
\begin{proof}
Write $w_1 = s_{i_1} \cdot s_{i_2} \cdot \ldots \cdot s_{i_l}$ as a reduced product of simple reflections.  For any subset $A = \{j_1, j_2, \dots, j_{|A|} \} \subseteq [l]$ ($j_1 < j_2 < \dots < j_{|A|}$), define $x_A = s_{i_{j_1}} \ldots s_{i_{j_{|A|}}}$. Also define the indicator function $I(P)$ of a statement $P$ to take the value $1$ if $P$ is true and $0$ if P is false.  By a direct calculation, the product $T_{w_1} T_{w_2} = T_{i_1} T_{i_2} \ldots T_{i_l} T_{w_2}$ is equal to
\begin{equation}\label{brute force equation1}
  \sum_{A \subseteq [l]} \xi^{l-|A|} T_{x_A w_2} \prod_{j \notin A} I\big( s_{i_j} \in L(x_{\{k \in A: k >j\}} w_2) \big),
\end{equation}
which implies the desired result.
\end{proof}

We will also use the important observation
\be
\label{e deg xi deg u}
\begin{array}{c}
\text{For any $f$ in $A$ that is a polynomial in $\xi$ with non-negative}\\  \text{integer coefficients, there holds $\deg_\xi(f) = \deg_\u(f)$.}
\end{array}
\ee

\subsection{} \label{main proof section}
\begin{proof}[Proof of Theorem \ref{main theorem}]
It is convenient to assume $v_1 \in W_a$, and this is possible because we can always write $v_1 = \pi v_1'$ with $\pi \in \Pi$ and $v_1' \in W_a$.  Then the theorem for $v_1'$ gives it for $v_1$ as $\C_{v_1 w_0 v_2}=\pi \C_{v_1' w_0 v_2} = \pi \lC_{v_1'} \C_{w_0} \rC_{v_2} = \lC_{v_1} \C_{w_0} \rC_{v_2}$.  This uses that $v_1'$ is primitive (Proposition \ref{pi mult prop}).

Begin by expanding out the product of canonical bases as follows:
\begin{equation}\label{factorization equation}
 \lC_{v_1} \C_{w_0} \rC_{v_2} = \big(\sum \lP_{x,v_1} T_x \big) \big(\sum_{y \in W_f} \uyw T_y \big) \big(\sum \rP_{z,v_2} T_z \big).
\end{equation}
The first sum is over $x \leq v_1$ such that $x\cdot w_0$ is reduced, and the third sum is over $z \leq v_2$ such that $w_0\cdot z$ is reduced.  The canonical basis element $\C_{v_1 w_0 v_2}$ is characterized by being bar invariant and equivalent to $T_{v_1 w_0 v_2} \mod \ui \L$.  Therefore, to prove the theorem it suffices to show that the only term in the expansion of (\ref{factorization equation}) not in $\ui \L$ is $\lP_{v_{1},v_{1}} T_{v_{1}} \ T_{w_0} \ \rP_{v_{2},v_{2}} T_{v_2} = T_{v_1 w_0 v_2}$ (which are equal by Proposition \ref{key_factorization_proposition}).  The proof takes four steps (A)--(D).  In (A) and (B) it is shown that a term in the expansion of (\ref{factorization equation}) is in $\ui \L$ provided $z$ is large.  The intuition is that it is only easier to get a large power of $\u$ if $z$ is large.  This is made precise by (C) and (D), which reduce the general case to the case $z$ is large.

\medskip{\it \noindent \emph{\textbf{(A)}}} Write $v_1 =  s_{i_1} \cdot s_{i_2} \cdot \ldots \cdot s_{i_l}$ as a reduced product of simple reflections.  Then for any $k \in \{0, \ldots, l\}$ and $y \in W_f$ with $y \neq w_0$,
\[ \uyw T_{s_{i_1} \ldots s_{i_k}} T_{s_{i_{k+1}} \ldots s_{i_l} y \lz} \in \ui \L\] provided $\lz$ is large with respect to $v_1$. In particular,
\[\uyw  T_{v_1} T_{y} T_{\lz} \in \ui \L.\]

\smallskip We will prove the main statement of (A) by induction on $\ell(w_0) - \ell(y)+ k$.  The case $\ell(w_0) - \ell(y) = 1$ and $k$ arbitrary holds because
$$T_{s_{i_1} \ldots s_{i_k}} T_{s_{i_{k+1}} \ldots s_{i_l} y \lz} = T_{s_{i_1} \ldots s_{i_k}} T_{s_{i_{k+1}} \ldots s_{i_l}} T_{y \lz} = T_{v_1}  T_{y \lz} = T_{v_1 y \lz} $$
by Propositions \ref{key_factorization_proposition} and \ref{partial primitive prop}.  The result is trivial for $k=0$. Now assume $k > 0$. If $s_{i_k} \notin L(s_{i_{k+1}} \ldots s_{i_l} y \lz)$, then
$$T_{s_{i_1} \ldots s_{i_k}} T_{s_{i_{k+1}} \ldots s_{i_l} y \lz} = T_{s_{i_1} \ldots s_{i_{k-1}}} T_{s_{i_{k}} \ldots s_{i_l} y \lz}$$ and we are done by induction.

Now suppose $s_{i_k} \in L(s_{i_{k+1}} \ldots s_{i_l} y \lz)$.  By Lemma \ref{long_z_lemma},
$$ s_{i_k} \ldots s_{i_l} y \lz = s_{i_{k+1}} \ldots s_{i_l} y' \lz',$$
where $\ell(y') > \ell(y)$, $y' \in W_f$, $\lz'$ large with respect to $v_1$, and $w_0 \cdot \lz'$ reduced.  Now compute
$$ T_{s_{i_k}} T_{s_{i_{k+1}} \ldots s_{i_l} y \lz} = \xi T_{s_{i_{k+1}} \ldots s_{i_l} y \lz} +
T_{s_{i_k}} \ldots s_{i_l} y \lz$$
$$ = \xi T_{s_{i_k} s_{i_{k+1}} \ldots s_{i_l} y' \lz'} +
T_{s_{i_k} \ldots s_{i_l} y \lz}.$$
Multiplying on the left by $\uyw T_{s_{i_1} \ldots s_{i_{k-1}}}$, we obtain
$$ \uyw T_{s_{i_1} \ldots s_{i_k}} T_{s_{i_{k+1}} \ldots s_{i_l} y \lz} = $$
$$ \uyw \xi T_{s_{i_1} \ldots s_{i_{k-1}}} T_{s_{i_{k}} \ldots s_{i_l} y' \lz'} +
\uyw T_{s_{i_1} \ldots s_{i_{k-1}}} T_{s_{i_{k}} \ldots s_{i_l} y \lz}.$$
The first term is in $\ui \L$ by induction since $(\u)^{\ell(y')-\ell(w_0)}$ has at least as high a power of $\u$ as $\uyw \xi$ (tracing this induction back to the base case involves at most $\ell(w_0)$ changes to $\lz$, so the largeness assumption remains valid).  The second term is in $\ui \L$ by the inductive statement with $k$ decreased by 1.


\smallskip{\it \noindent \emph{\textbf{(B)}}} The product $\uyw \lP_{x, v_1} T_{x}  T_y T_{\lz} \in \ui \L$ for $x < v_1$, $x\cdot w_0$ reduced, $y \in W_f$, and $\lz$ large with respect to $x$.

\smallskip The proof is the same as that for (A)  except that the base case is for $y = w_0$, which holds since $\lP_{x,v_1}$ is a polynomial in $\ui$ with no constant term and $x \cdot w_0 \cdot \lz$  is reduced by Proposition \ref{key_factorization_proposition}.

\smallskip{\it \noindent \emph{\textbf{(C)}}} Given $x, \sz \in W_e$ such that $w_0 \cdot \sz$ is reduced, there exists a $v \in W_e$ so that $\lzy := \sz v$ is large with respect to $x$ and $\sz \cdot v$ and $w_0 \cdot \lzy$ are reduced factorizations.

Choose any $\lambda$ such that $\langle \lambda, \alpha_i \rangle >> \ell(\sz) + \ell(x) + \ell(w_0)$ for $i \in [n]$ and put $v:=\sz^{-1} \lzy$.  We will use Lemma \ref{reducedlemma} to show that $\sz \cdot v$ is reduced.  It is convenient to instead show $v^{-1} \cdot \sz^{-1}$ is reduced. Suppose $\alpha + k\delta \in R_+$ with $\alpha \in \fcoroots$ and that $\sz^{-1}(\alpha+k\delta) \in R_-$.  Then $\sz^{-1}(\alpha) \in R_-$, and because the product $w_0 \cdot \sz$ is reduced, we must have $\alpha \in \fcoroots_-$ (a similar fact was shown in Proposition \ref{key_factorization_proposition}).   Now $v^{-1} \sz^{-1} (\alpha+k\delta) = \alpha + (k+ \langle \lambda, \alpha \rangle) \delta$.  The integer $k+ \langle \lambda, \alpha \rangle << 0$ because $\sz^{-1}(\alpha+k\delta) \in R_-$ implies $k$ is bounded by a constant times $\ell(\sz)$.  Therefore $v^{-1} \sz^{-1} (\alpha+k\delta) \in R_-$, as desired.

\smallskip{\it \noindent \emph{\textbf{(D)}}} Any term $\lP_{x, v_1} \uyw \rP_{z, v_2} T_x T_{y \sz} $ from the expansion of (\ref{factorization equation}) is in $\ui \L$.

\smallskip
Choose $\lz$ large with respect to $x$, as was shown to exist in (C), such that there exists $v$ so that $\lz = z \cdot v$ is reduced. Then compute
\be T_x T_{y\lz} = T_x T_{y \sz} T_v = \left(\sum_{a \in W_e} f_{x,y \sz,a} T_a\right) T_v = \sum_{b \in W_e} \left(\sum_{a \in W_e} f_{x,y \sz,a} f_{a,v,b}\right)T_b. \ee
By (\ref{e deg xi deg u}), the highest power of $\u$ occurring in $T_x T_{y \sz}$ is  $\max_a(\deg_\xi(f_{x,y \sz,a}))$. Let $a' \in W_e$ be an element with  $f_{x,y \sz,a'}$ realizing this maximum degree and $b' \in W_e$ an element with  $f_{a',v,b}$ nonzero. Then since the $f$'s are polynomials in $\xi$ with non-negative coefficients (Lemma \ref{l structure coefficients}), \be \deg_\xi(f_{x,y\sz,a'})) \leq \deg_\xi \left(\sum_{a \in W_e} f_{x,yz,a}f_{a,v,b'}\right).\ee  Moreover, again by (\ref{e deg xi deg u}), the right-hand side of this inequality is the $\u$-degree of the coefficient of  $T_{b'}$ in $T_x T_{y\lz}$.
Thus $\lP_{x, v_1} \uyw T_{x} T_{y \lz}$ in $\ui \L$ (by (A) and (B)) implies the same for $\lP_{x, v_1} \uyw \rP_{z, v_2} T_{x} T_{y \sz}$, as desired.
\end{proof}

\section{Concluding remarks}

We have tried to find an analog of Theorem \ref{main theorem} for $w$ not in $W_{(\nu)}$, or just in the finite Weyl group setting replacing $w_0$ with the longest element $\lj{w_0}{J}$ of some parabolic subgroup $J$. One can ask, for instance, for which $v_1,v_2\in W_f$ the identity $\C_{v_1 \lj{w_0}{J} v_2}=\lC_{v_1}\C_{w_0^J}\rC_{v_2}$ holds. We concluded after a cursory investigation that this holds so rarely that it wouldn't be of much use. We could certainly have overlooked something, but it's more likely that a nice extension of this result requires that a factorization like the above holds but only after quotienting by some submodule spanned by a subset of the canonical basis.

\section{Appendix}

\begin{lemma*}
For any $x, y \in W_e$, $x \cdot y$ is a reduced factorization if and only if $x(y(R_{+}) \cap R_{-}) \subseteq R_{-}$.
\end{lemma*}
\begin{proof}
Given $\alpha \in R$, there are eight possibilities for the signs of $\alpha$, $y(\alpha)$, and $xy(\alpha)$.  Let
$$N_{\epsilon_1 \epsilon_2 \epsilon_3} = |\{\alpha \in R : xy(\alpha) \in R_{\epsilon_1}, \; y(\alpha) \in  R_{\epsilon_2}, \; \alpha \in  R_{\epsilon_3} \}|$$
where $\epsilon_i \in \{+, -\}$.    With this notation, $x(y(R_{+}) \cap R_{-}) \cap R_+$ has cardinality $N_{+ - +}$, so the condition $x(y(R_{+}) \cap R_{-})\subseteq R_{-}$ is equivalent to $N_{+ - +} = 0$.
We have
\begin{equation*}
\begin{array}{lll}
\ell(xy) & = & N_{- - +} + N_{- + +}, \\
\ell(y) & = & N_{- - +} + N_{+ - +}, \\
\ell(x) & = & N_{- + +} + N_{- + -}.
\end{array}
\end{equation*}
A root $\alpha$ contributes to $N_{- + -}$ if and only if $-\alpha$ contributes to $N_{+ - +}$.  Hence
$$ \ell(x) + \ell(y) = N_{- - +} +N_{+ - +} + N_{- + +} + N_{- + -} = \ell(xy) + 2N_{+ - +}.$$
Therefore $x \cdot y$ is a reduced factorization if and only if $N_{+ - +} = 0$.
\end{proof}

\begin{lemma*}
Suppose $a,y\in W_f$ and $\alpha\in \fcoroots_+$.
\setcounter{ctr}{0}
\begin{list}{\emph{(\roman{ctr})}} {\usecounter{ctr} \setlength{\itemsep}{1pt} \setlength{\topsep}{2pt}}
\item If $s_{\alpha}a < a$, then $s_{\alpha}ay < ay \iff a^{-1}s_{\alpha}ay > y$.
\item If $s_{\alpha}a > a$, then $s_{\alpha}ay > ay \iff a^{-1}s_{\alpha}ay > y$.
\end{list}
\end{lemma*}
\begin{proof}
    For (i), we have $s_{\alpha} a < a$ implies $a^{-1}(\alpha) \in \fcoroots_-$ by  (\ref{e hyperplane separate}). The element $a^{-1}s_{\alpha}a$ is a reflection $s_{\beta}$ where $\beta=-a^{-1}(\alpha) \in \fcoroots_+$. Now compute
    \[y^{-1}(\beta)=y^{-1}(-a^{-1}(\alpha))=-(ay)^{-1}(\alpha).\]
    Hence
    \[s_{\alpha} ay < ay \iff -(ay)^{-1}(\alpha)\in \fcoroots_+ \iff y^{-1}(\beta) \in \fcoroots_+ \iff s_{\beta}y>y,\]
    where the first and last equivalence again use (\ref{e hyperplane separate}).
    The proof of (ii) is the same except with $\beta$ defined to be $a^{-1}(\alpha) \in \fcoroots_+$ instead of $-a^{-1}(\alpha)$.
\end{proof}

\section*{Acknowledgments}
This paper would not have been possible without the generous advice from and many detailed discussions with Mark Haiman.  I am also grateful to Michael Phillips, Ryo Masuda, and Kristofer Henriksson for help typing and typesetting figures.

 \end{document}